\documentclass{article}
\usepackage{subfig}
\usepackage[section] {placeins}
\usepackage{amsthm}
\usepackage{amssymb,amsmath, mathdots}
\usepackage{anysize}
\usepackage{psfrag}
\usepackage{url}
\newtheorem{prob}{Problem}[section]
\newtheorem{theorem}{Theorem}[section]

\newtheorem{proposition}{Proposition}[section]
\newtheorem{cor}{Corollary}[section]

\newtheorem{lem}{Lemma}[section]

\numberwithin{table}{section}
\begin{document}

\title{Infinitely many prime knots with the same Alexander invariants.}

\author{Louis H. Kauffman\\
Department of Mathematics, Statistics and Computer Science, \\
University of Illinois at Chicago, 851 S. Morgan St., \\
Chicago IL 60607-7045, USA \\
\texttt{kauffman@uic.edu}\\
and\\
Pedro Lopes\\
Center for Mathematical Analysis, Geometry, and Dynamical Systems, \\
Department of Mathematics, \\
Instituto Superior T\'{e}cnico, Universidade de Lisboa\\
1049-001 Lisbon, Portugal \\
\texttt{pelopes@math.tecnico.ulisboa.pt}}
\date{March 18, 2017}
\maketitle

\begin{abstract}
We  revisit the issue of the existence of infinitely many distinct prime knots with the same Alexander invariant. We present infinitely many distinct families, each family made up of infinitely many distinct knots. Within each family, the Alexander invariant is the same. Unlike other examples in the literature, ours are elementary and based on a sub-collection of pretzel knots with three tassels.
\end{abstract}



\section{Introduction}\label{sec:intro}

In  this article we present infinitely many families of infinitely many distinct prime knots with the following property. Each of these families is characterized by its knots all having the same elementary ideals.  We prove the knots are distinct by calculating their Jones polynomial.

Seifert showed, by using the Seifert pairing and putting appropriate bands on a disk, that one could obtain any specified Seifert matrix. Hence one could obtain any Alexander module (in later language), \cite{Seifert}. It was not in Seifert's power to prove that this method produced infinitely many knots with the same Seifert matrix, but he did show how to produce
non-trivial knots with unit Alexander polynomial and pairs of knots with the same Alexander polynomial. Seifert's knotted band and twist method surely produces infinitely many prime knots with the same Seifert pairing. We will reserve proving that for a sequel to the present paper (due to the technicalities of applying the Jones polynomial to this situation). In the present paper we will use a different approach via pretzel knots that produces infinitely many prime knots with the same sequence of elementary ideals.

Techniques for constructing infinite families of distinct knots with the same Alexander polynomials, or the same Alexander modules, have been known for some time.  However, the knots that arise from these constructions tend to be quite complicated as well as the proofs of the corresponding facts (\cite{AitchisonSilver}, \cite{Friedl}, \cite{HittSilver}, \cite{Livingston1981}, \cite{Livingston2002}, \cite{Morton}, \cite{Sto}).  Here we will demonstrate that there are infinite sets of  pretzel knots with three tassels, with isomorphic Alexander modules.  These have the property that all the knots are genus one, prime, and hyperbolic.   We discriminate these knots by calculating their Jones polynomials. We are aware of the telling apart of pretzel knots by other methods (\cite{Kawauchi-Pretzels}) but, here again, the proofs are longer and more involved than our straightforward calculations of the Jones polynomial via skeining.

\bigbreak

In the following we will let the expression ``knot'' stand for knot or link except where otherwise noted.

Consider the following problem.
\begin{prob}\label{prob:prob1}
Find an infinite  family $\cal F$ of distinct knots, all having the same Alexander polynomial.
\end{prob}


Due to  the fact that the Alexander polynomial is multiplicative with respect to connected sums of knots, it is easy to construct an infinite sequence of knots with the same Alexander polynomial, given a knot with the indicated Alexander polynomial. It is a matter of making connected sums of the given knot with non-trivial knots with trivial Alexander polynomial, like the Kinoshita-Terasaki knot, or Conway's knot. The problem becomes more interesting when we try to solve it over more specific subfamilies of knots. To the best of our knowledge, the first such attempt is \cite{Morton} where an infinite family of distinct fibered knots with the same Alexander polynomial is presented. In \cite{Sto}  Problem \ref{prob:prob1} is solved for generic Alexander polynomial with the solutions being provided over the family of arborescent knots and with a Table displaying the state-of-the-art concerning the existence or not of solutions when we restrict the solutions to the family of fibered knots, knots with fixed genus, knots with fixed number of components, etc.

In this article we generalize Problem \ref{prob:prob1} to the following.

\begin{prob}\label{prob:prob2}
Find an infinite  family  $\cal F$ of distinct knots, all having the same sequence of elementary ideals.
\end{prob}

In this article we solve Problem \ref{prob:prob2} for infinitely many families of knots. Specifically, we present infinitely many families of infinitely many distinct prime knots of genus $1$ with the same Alexander polynomial, per family. The Alexander polynomial at issue is $A - (2A-1)t + At^2$ with  different $A$'s for different families. Given a positive integer $s$,  it is realized by the family of pretzel knots of the sort $P(-2s-1, 2s+1, 2i+1)$ for each positive integer $i>s+3/2$, where $A=-s(s+1)$. Furthermore, within each of these families, we single out an infinite subfamily whose knots all have the same  sequence of elementary ideals. Moreover, the issue of telling these knots apart is simply solved here by  calculating the corresponding Jones polynomials via skein theory.

We start by stating and proving the following features of these knots.

\begin{proposition}
Given integers $i>s\geq 1$, the pretzel knot $P(-2s-1, 2s+1, 2i+1)$ is a prime knot of genus $1$, hyperbolic and not fibered.
\end{proposition}
\begin{proof}
Leaning on Corollary $2.7$, statement $2.$ in \cite{KL}, we know the knots under consideration are all of genus $1$. Since the genus is additive under connected sums, we conclude that these knots are all prime. Leaning on Theorem $2.2$ in \cite{KL}, we know these knots are hyperbolic. Furthermore, since the Alexander polynomials of $P(-2s-1, 2s+1, 2i+1)$ (\cite{Lickorish}, pages 56 and 57) are not monic, then these knots are not fibered, \cite{Rapaport}.
\end{proof}

The referee drew our attention to the article by Landvoy (\cite{Landvoy}). A short description of this article now follows. The purpose of \cite{Landvoy} is to present expressions for the calculation of the Jones polynomial of pretzel knots by way of the Kauffman bracket. Closed form formulas are obtained for the pretzel knots with three tassels. These formulas enable one to tell apart the elements of a family of pretzel knots on three tassels whose Alexander polynomials are trivial. The triviality of these Alexander polynomials had already been flagged by Parris in his PhD thesis (\cite{Parris}).  On the other hand, in the current article we set out to obtain families of pretzel knots in three tassels with the same Alexander ideals; also, the Alexander polynomials of our examples are not trivial. Furthermore, we use the Jones polynomial to tell apart the individual pretzel knots we work with. Finally, we calculate the Jones polynomial directly via skeining.

\bigbreak

The rest of this article is organized as follows. In Section \ref{sec:exampleAlex} we present a first example of a family of pretzel knots with the same Alexander polynomial and extract an infinite subfamily whose knots all have the same sequence of elementary ideals. In Section \ref{sec:exampleJones} we show that the knots from the family presented in Section \ref{sec:exampleAlex} are told apart by their Jones polynomial. In Section \ref{sec:infty} we show that the results of Sections \ref{sec:exampleAlex} and \ref{sec:exampleJones} generalize to infinitely many families of knots. In Section \ref{sec:furtherwork} we point out directions for future work.

\section{The example. Calculating the Alexander invariants.}\label{sec:exampleAlex}

Consider the pretzel knots in Figure \ref{fig:pretzelfamily}.

\begin{figure}[!ht]
	\psfrag{P(-3, 3, 3)}{\Huge$P(-3, 3, 3)$}
	\psfrag{P(-3, 3, 2i+1)}{\Huge$P(-3, 3, 2i+1)$}
	\psfrag{2i+1}{\huge$2i+1$}
	\psfrag{eq}{\Huge$c=2b-a$}
	\centerline{\scalebox{.5}{\includegraphics{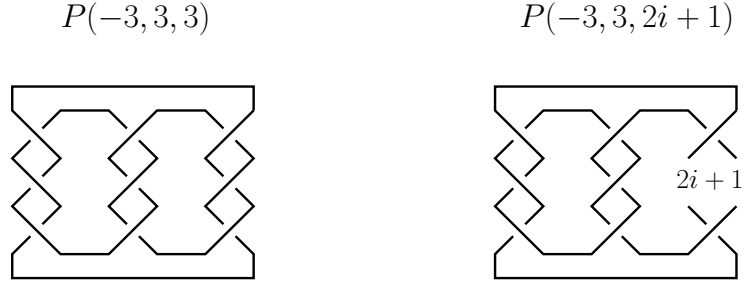}}}
	\caption{The pretzels knots $P(-3, 3, 3)$ and $P(-3, 3, 2i+1)$ for non-zero integer $i$.}\label{fig:pretzelfamily}
\end{figure}

\begin{theorem}\label{thm:exampleAlex}
Let $i$ be a positive integer. The knots from the family $\big(  P(-3, 3, 2i+1) \big)_{i\in \mathbf{Z}^{+}}$ have the same Alexander polynomial, $2t^2 - 5t + 2$.
Furthermore, if we choose $i=i_k=3k-1$, then the knots from the subfamily $\big(  P(-3, 3, 2i_k+1) \big)_{k\in \mathbf{Z}^+}$ have the same elementary ideals.
\end{theorem}
\begin{proof}
From the formulas in \cite{Lickorish}, pages 56 and 57, we obtain the following Seifert matrices for $P(-3, 3, 3)$ and $P(-3, 3, 2i+1)$, respectively
\[
\begin{pmatrix}
0 & 2\\
1 & 3
\end{pmatrix} \qquad \text{ and } \qquad \begin{pmatrix}
0 & 2\\
1 & i+2
\end{pmatrix}
\]
from which we obtain the presentation matrices for the Alexander modules of $P(-3, 3, 3)$ and $P(-3, 3, 2i+1)$, respectively,
\[
\begin{pmatrix}
0 & 2t-1\\
t-2 & 3t-3
\end{pmatrix} \qquad \text{ and } \qquad \begin{pmatrix}
0 & 2t-1\\
t-2 & (i+2)t-(i+2)
\end{pmatrix}
\]

Denoting the Alexander polynomial by $\Delta$, we obtain:
\[
\Delta(P(-3, 3, 3)) \doteq \det \begin{pmatrix}
0 & 2t-1\\
t-2 & 3t-3
\end{pmatrix} = -(2t-1)(t-2)
\]
and
\[
\Delta(P(-3, 3, 2i+1)) \doteq \det \begin{pmatrix}
0 & 2t-1\\
t-2 & (i+1)t-(i+1)
\end{pmatrix} = -(2t-1)(t-2)
\]
Furthermore, since $3t-3 = (2t-1) + (t-2)$, the $2$nd elementary ideal of $P(-3, 3, 3)$ is generated by $t-2$ and $2t-1$. On the other hand, if we choose $i=i_k=3k-1$ for any positive integer $k$, then $(i_k+1)t-(i_k+1) = k(2t-1 + t-2)$. Hence, the knots from the subfamily $\big(  P(-3, 3, 2i_k+1) \big)_{k\in \mathbf{Z}^+}$  have the same $2$nd elementary ideals, which are generated by $t-2$ and $2t-1$.

This concludes the proof.
\end{proof}

\section{The example. Calculating the Jones invariant.}\label{sec:exampleJones}

We now prove that the knots from the family $\big(  P(-3, 3, 2i+1) \big)_{i\in \mathbf{Z}^{+}, i>2}$ are all distinct by showing that no two of them have the same Jones polynomial. We resort to the skein relations in order to compute the Jones polynomial. See, for instance \cite{Lickorish}, Proposition $3.7$ on page 28. Denoting the Jones polynomial by $V$, we have that the Jones polynomial of the unknot is identically $1$, and that
\[
t^{-1}V(L_{+}) - t V(L_{-}) + (t^{-1/2}-t^{1/2})V(L_{0}) = 0
\]
where $L_{+}$, $L_{-}$, and $L_{0}$ stand for $3$ diagrams which are everywhere the same but for the neighborhoods indicated in Figure \ref{fig:skeindiags}.

\begin{figure}[!ht]
	\psfrag{+}{\Huge$L_{+}$}
	\psfrag{-}{\Huge$L_{-}$}
	\psfrag{0}{\huge$L_{0}$}
	\psfrag{eq}{\Huge$c=2b-a$}
	\centerline{\scalebox{.5}{\includegraphics{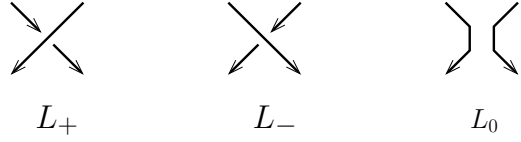}}}
	\caption{The differences in diagrams $L_+$, $L_-$, and $L_0$ occur at the indicated neighborhoods.}\label{fig:skeindiags}
\end{figure}

Lemma \ref{lem:T(2, 2k)} is an auxiliary result which is probably already known. We include it here for completeness.

\begin{lem}\label{lem:T(2, 2k)}
Given a positive integer $k>1$, The Jones polynomial of the torus link of type $(2, 2k)$ with the orientation shown in Figure \ref{fig:sigma1-4} is
\[
V(T(2, 2k)) = -t^{-1/2}\bigg(  t^{-2k}  + t^{-2k+2} + \sum_{i=0}^{2k-3} (-1)^i t^{-i}\bigg)
\]
\end{lem}
\begin{figure}[!ht]
	\psfrag{L+}{\Huge$L_{+}$}
	\psfrag{L-}{\Huge$L_{-}$}
	\psfrag{L0}{\huge$L_{0}$}
	\psfrag{eq}{\Huge$c=2b-a$}
	\centerline{\scalebox{.5}{\includegraphics{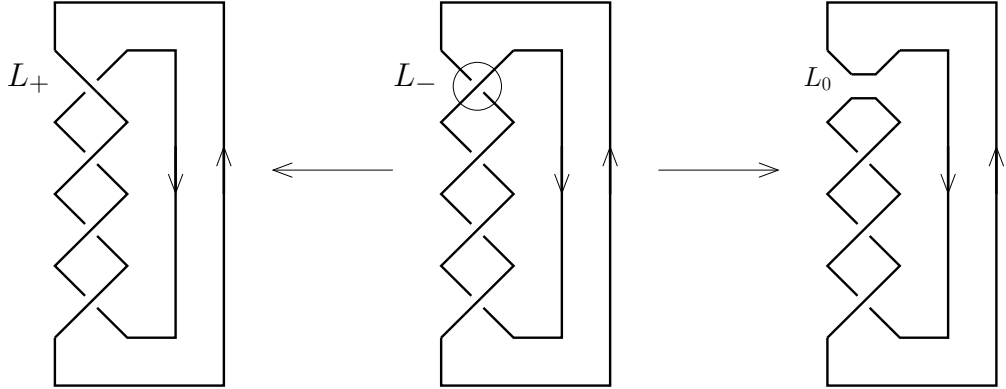}}}
	\caption{Torus link of type $(2, 4)$  in the center diagram. We begin the skeining at the crossing surrounded by the circle. The leftmost diagram is  the Hopf link; the rightmost diagram is the unknot.}\label{fig:sigma1-4}
\end{figure}
\begin{proof}
We leave it for the reader to prove that the Jones polynomial of the Hopf Link with the orientation illustrated in Figure \ref{fig:sigma1-4} is $-t^{-5/2}-t^{-1/2}$.

With the remarks from the caption in Figure \ref{fig:sigma1-4} we conclude that
\[
t^{-1}(-t^{-5/2}-t^{-1/2}) - tV(T(2, 4)) + (t^{-1/2}-t^{1/2})\cdot 1 = 0
\]
so
\[
V(T(2, 4)) = t^{-2}(-t^{-5/2}-t^{-1/2}) + t^{-1}(t^{-1/2}-t^{1/2}) = -t^{-4-1/2} - t^{-2-1/2} + t^{-1-1/2} - t^{-1/2}
\]

Now for the inductive step.
\begin{figure}[!ht]
	\psfrag{L+}{\Huge$L_{+}$}
	\psfrag{L-}{\Huge$L_{-}$}
	\psfrag{L0}{\huge$L_{0}$}
	\psfrag{2k+2}{\Huge$2k+2$}
	\centerline{\scalebox{.5}{\includegraphics{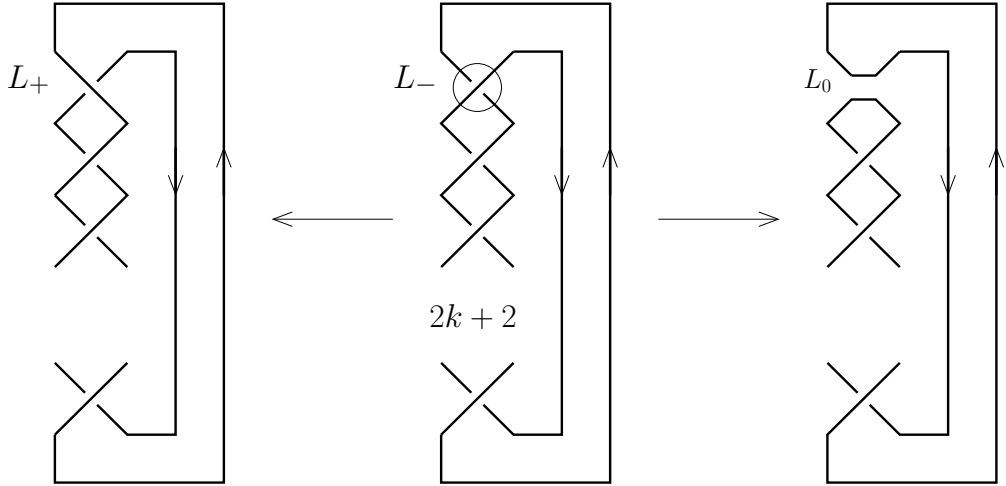}}}
	\caption{Torus link of type $(2, 2k+2)$ in the center diagram. Skeining at the crossing surrounded by the circle we obtain the following. The leftmost diagram is  the $T(2, 2k)$ link; the rightmost diagram is the unknot.}\label{fig:sigma1-2k}
\end{figure}
With the remarks from the caption of Figure \ref{fig:sigma1-2k} we obtain
\[
t^{-1}V(T(2, 2k)) - tV(T(2, 2k+2)) + (t^{-1/2}-t^{1/2})\cdot 1 = 0
\]
so
\begin{align*}
&V(T(2, 2k+2)) = t^{-2}(-t^{-1/2})\bigg(  t^{-2k}  + t^{-2k+2} + \sum_{i=0}^{2k-3} (-1)^i t^{-i}\bigg)  + t^{-1}(t^{-1/2}-t^{1/2}) = \\
& = -t^{-1/2}\bigg(  t^{-2k-2}  + t^{-2k} + \sum_{i=0}^{2k-3} (-1)^i t^{-i-2}   - t^{-1} + 1 \bigg) = \\
& = -t^{-1/2}\bigg(  t^{-2k-2}  + t^{-2k} + \sum_{i'=2}^{2k-1} (-1)^{i'-2} t^{-i'}   - t^{-1} + 1 \bigg) =\\
& = -t^{-1/2}\bigg(  t^{-(2k+2)}  + t^{-(2k+2)-2} + \sum_{i'=0}^{2k+2-3} (-1)^{i'} t^{-i'}   \bigg)
\end{align*}
This concludes the proof.
\end{proof}

\begin{theorem}\label{thm:Pretzelknots}
Let $i>2$ be an integer. The Jones polynomial of the pretzel knot $P(-3, 3, 2i+1)$ is
\[
t^{-2i-4} - t^{-2i-3} + t^{-2i-2} - 2t^{-2i-1} + t^{-2i} - t^{-2i+1} + t^{-2i+2} + 1
\]
\end{theorem}
\begin{proof}
\begin{figure}[!ht]
	\psfrag{L+}{\Huge$L_{+}$}
	\psfrag{L-}{\Huge$L_{-}$}
	\psfrag{L0}{\huge$L_{0}$}
	\psfrag{2i+1}{\Huge$2i+1$}
	\psfrag{P(-3, 3, 2i+1)}{\Huge$P(-3, 3, 2i+1)$}
	\centerline{\scalebox{.5}{\includegraphics{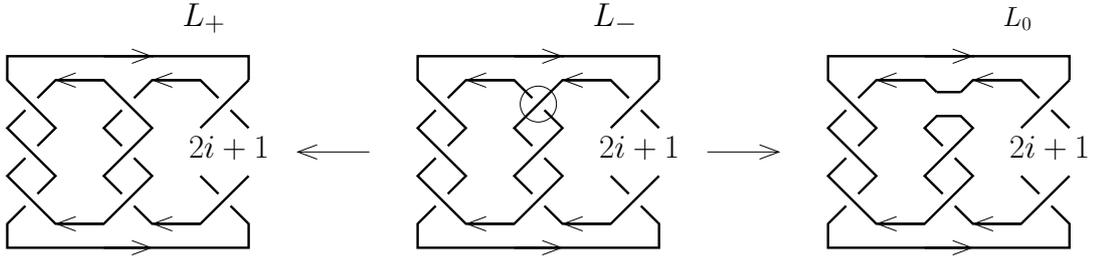}}}
	\caption{The center diagram is $P(-3, 3, 2i+1)$ with the indicated orientation. Skeining at the crossing surrounded by the circle we obtain the following. The leftmost diagram is  the $P(-3, 1, 2i+1)$; the rightmost diagram is the $T(2, 2i-2)$ link for which Lemma \ref{lem:T(2, 2k)} already provides the Jones polynomial.}\label{fig:skeiningpretzels}
\end{figure}
\begin{figure}[!ht]
	\psfrag{L+}{\Huge$L_{+}$}
	\psfrag{L-}{\Huge$L_{-}$}
	\psfrag{L0}{\huge$L_{0}$}
	\psfrag{2i+1}{\Huge$2i+1$}
	\psfrag{P(-3, 1, 2i+1)}{\Huge$P(-3, 1, 2i+1)$}
	\centerline{\scalebox{.5}{\includegraphics{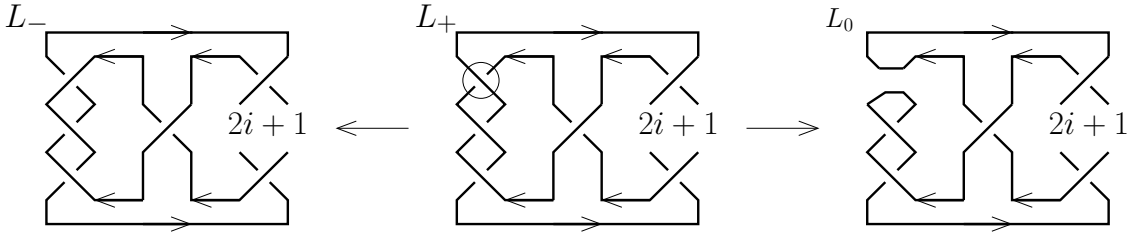}}}
	\caption{The center diagram is $P(-3, 1, 2i+1)$ with the indicated orientation. Skeining at the crossing surrounded by the circle we obtain the following. The leftmost diagram is  the unknot; the rightmost diagram is the $T(2, 2i+2)$ link for which Lemma \ref{lem:T(2, 2k)} already provides the Jones polynomial.}\label{fig:skeiningpretzelscontd}
\end{figure}

From Figure \ref{fig:skeiningpretzelscontd} we have
\begin{align}\label{eq:1}
t^{-1}V(P(-3, 1, 2i+1)) = t\cdot 1 + (t^{1/2}-t^{-1/2})(-t^{-1/2})\bigg(  t^{-2i-2}  + t^{-2i} + \sum_{j=0}^{2i-1} (-1)^j t^{-j}\bigg)
\end{align}

From Figure \ref{fig:skeiningpretzels} we have
\[
tV(P(-3, 3, 2i+1)) = t^{-1}V(P(-3, 1, 2i+1)) + (t^{-1/2}-t^{1/2})(-t^{-1/2})\bigg(  t^{-2i+2}  + t^{-2i+4} + \sum_{j=0}^{2i-5} (-1)^j t^{-j}\bigg)
\]
and composing with Equation \ref{eq:1}, we get
\begin{multline*}
V(P(-3, 3, 2i+1)) = 1 + t^{-1}(t^{1/2}-t^{-1/2})(-t^{-1/2})\bigg(  t^{-2i-2}  + t^{-2i} + \sum_{j=0}^{2i-1} (-1)^j t^{-j}\bigg) + \\
+ t^{-1}(t^{-1/2}-t^{1/2})(-t^{-1/2})\bigg(  t^{-2i+2}  + t^{-2i+4} + \sum_{j=0}^{2i-5} (-1)^j t^{-j}\bigg) = \\
= 1 + (t^{-2}-t^{-1})\bigg(  t^{-2i-2}  + t^{-2i} - t^{-2i+2}  - t^{-2i+4} + \sum_{j=0}^{2i-1} (-1)^j t^{-j} - \sum_{j=0}^{2i-5} (-1)^j t^{-j}\bigg) = \\
= 1 + (t^{-2}-t^{-1})\bigg(  t^{-2i-2}  + t^{-2i} - t^{-2i+2}  - t^{-2i+4} + \sum_{j=2i-4}^{2i-1} (-1)^j t^{-j} \bigg) = \\
= 1 + (t^{-2}-t^{-1})\bigg(  t^{-2i-2}  + t^{-2i} - t^{-2i+2}  - t^{-2i+4} - t^{-2i+1} + t^{-2i+2} - t^{-2i+3} + t^{-2i+4} \bigg) = \\
= 1 + (t^{-2}-t^{-1})\bigg(  t^{-2i-2}  + t^{-2i}   - t^{-2i+1}  - t^{-2i+3}  \bigg) = \\
= 1 +  t^{-2i-4} - t^{-2i-3} + t^{-2i-2}- t^{-2i-1}   - t^{-2i-1} + t^{-2i}- t^{-2i+1} +t^{-2i+2} = \\
= t^{-2i-4} - t^{-2i-3} + t^{-2i-2}-2t^{-2i-1}  + t^{-2i}- t^{-2i+1} +t^{-2i+2} + 1
\end{multline*}
This concludes the proof.
\end{proof}

\begin{cor}
The family of knots $\big(  P(-3, 3, 2i+1) \big)_{i\in \mathbf{Z}^{+}, i>2}$ is made up of distinct knots.
\end{cor}
\begin{proof}
This is a direct consequence of Theorem \ref{thm:Pretzelknots}.
\end{proof}

\section{An infinite family of infinite families}\label{sec:infty}

\subsection{Calculating the Alexander invariants.}\label{subsec:inftyAlex}

\begin{theorem}
For each positive integer $s$, the knots from the family $\big(  P(-(2s+1), 2s+1, 2i+1) \big)_{i > s}$ have the same Alexander polynomial, $s(s+1)t^2 - (2s(s+1)+1)t + s(s+1)$.
Furthermore, if we choose $i=i_{k, s}=(2k-1)s + k-1$, then the knots from the subfamily $\big(  P(-(2s+1), 2s+1, 2i_{k, s} + 1) \big)_{k\in \mathbf{Z}^+}$ have the same sequence of elementary ideals.
\end{theorem}
\begin{proof}
The proof is easily adapted from the proof of Theorem \ref{thm:exampleAlex}
\end{proof}

\begin{figure}[!ht]
	\psfrag{L+}{\Huge$L_{+}$}
	\psfrag{L-}{\Huge$L_{-}$}
	\psfrag{L0}{\Huge$L_{0}$}
	\psfrag{2i+1}{\huge$2i+1$}
	\psfrag{2s+1}{\huge$2s+1$}
	\psfrag{2s-1}{\huge$2s-1$}
	\psfrag{P(-3, 3, 2i+1)}{\Huge$P(-(2s+1), 2s+1, 2i+1)$}
	\centerline{\scalebox{.5}{\includegraphics{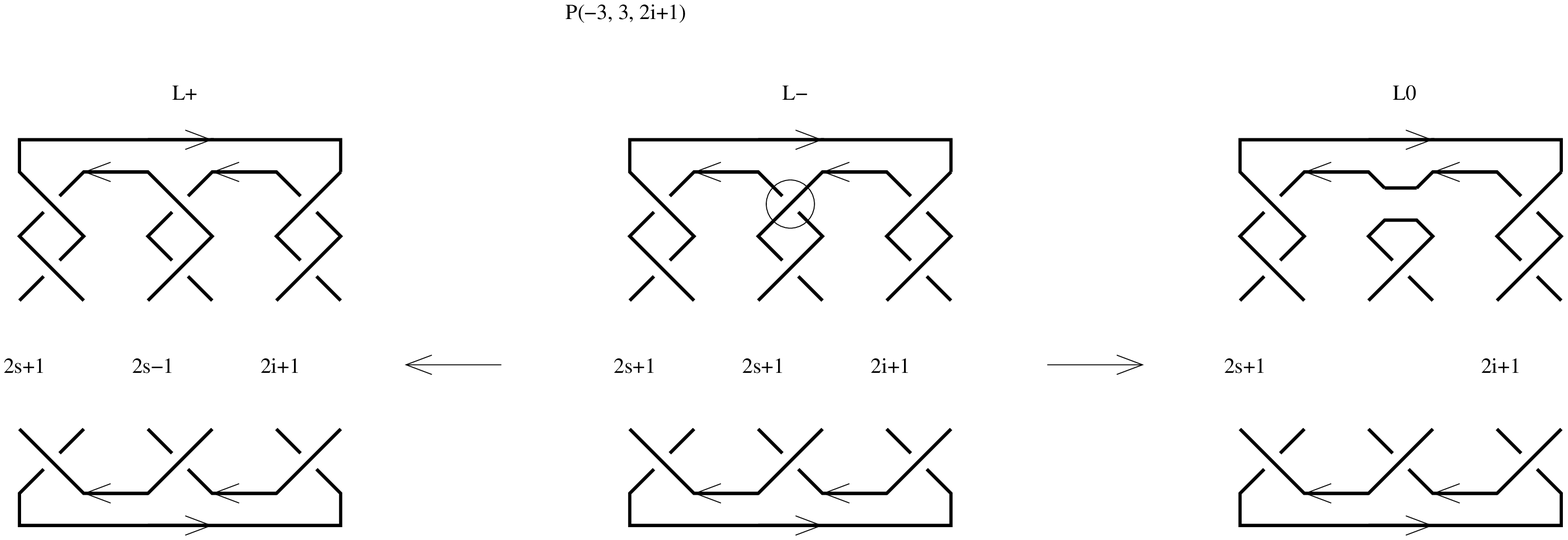}}}
	\caption{The center diagram is $P(-(2s+1), 2s+1, 2i+1)$ with the indicated orientation. Skeining at the crossing surrounded by the circle we obtain the following. The leftmost diagram is  the $P(-(2s+1), 2s-1, 2i+1)$; the rightmost diagram is the $T(2, 2i-2s)$ link for which Lemma \ref{lem:T(2, 2k)} already provides the Jones polynomial.}\label{fig:skeiningpretzels-s}
\end{figure}

\subsection{Calculating the Jones invariant.}\label{subsec:inftyJones}

\begin{figure}[!ht]
	\psfrag{L+}{\Huge$L_{+}$}
	\psfrag{L-}{\Huge$L_{-}$}
	\psfrag{L0}{\huge$L_{0}$}
	\psfrag{2i+1}{\huge$2i+1$}
	\psfrag{2s+1}{\huge$2s+1$}
	\psfrag{2s-1}{\huge$2s-1$}
	\psfrag{P(-3, 1, 2i+1)}{\Huge$P(-(2s+1), 2s-1, 2i+1)$}
	\centerline{\scalebox{.5}{\includegraphics{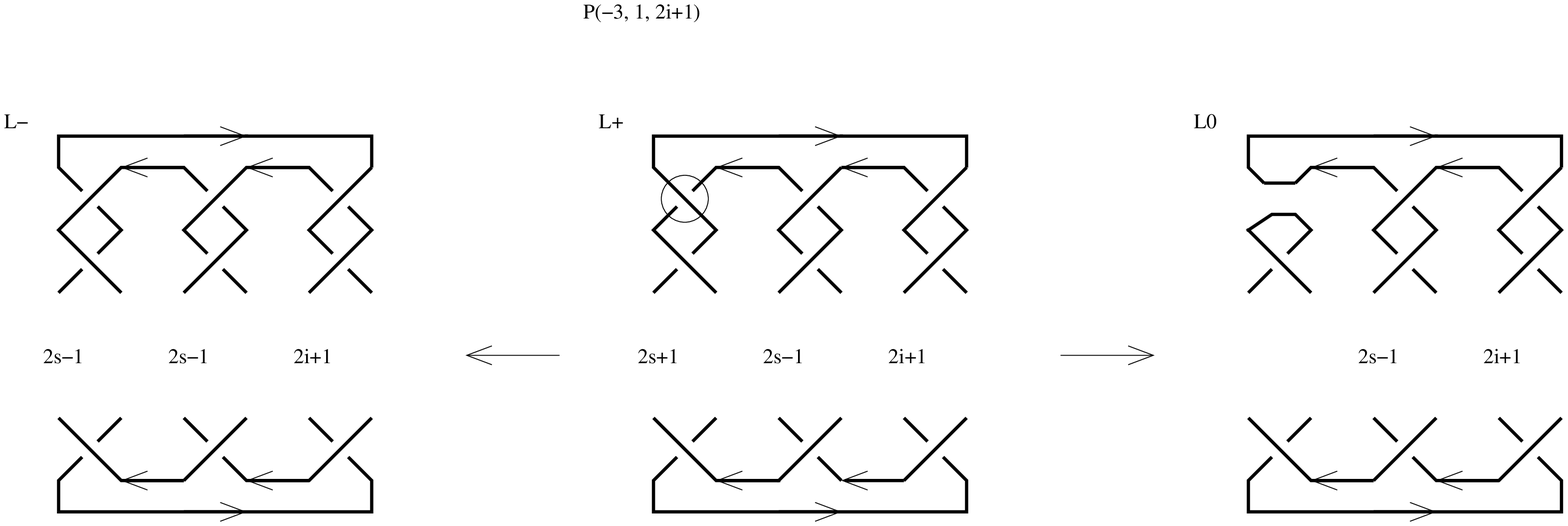}}}
	\caption{The center diagram is $P(-(2s+1), 2s-1, 2i+1)$ with the indicated orientation. Skeining at the crossing surrounded by the circle we obtain the following. The leftmost diagram is $P(-(2s-1), 2s-1, 2i+1)$; the rightmost diagram is the $T(2, 2i+2s)$ link for which Lemma \ref{lem:T(2, 2k)} already provides the Jones polynomial.}\label{fig:skeiningpretzelscd-s}
\end{figure}

\begin{theorem}\label{thm:Pretzelknotsinfty}
Let $i>s+3/2\geq 5/2$ be positive integers. The Jones polynomial of the pretzel knot $P(-(2s+1), 2s+1, 2i+1)$ is
\begin{align}\label{form:form}
1 + t^{-2i-1}\bigg[  t^{2s+1} + t^{-(2s+1)}  - ( t + t^{-1} )  - 2s + \sum_{j=1}^{2s}(-1)^{j+1}\, (2s+1-j)\, ( t^{j} + t^{-j} ) \bigg]
\end{align}
\end{theorem}
\begin{proof}
We begin by establishing the following recurrence relation with the help of Figures \ref{fig:skeiningpretzels-s} and \ref{fig:skeiningpretzelscd-s}. These figures illustrate how to skein starting from $P(-(2s+1), 2s+1, 2i+1)$ in the indicated crossings and are an adaptation of the argument and calculations in Theorem \ref{thm:Pretzelknots}; the role of the trivial knot there corresponds here to $P(-(2s-1), 2s-1, 2i+1 )$. Specifically, skeining where indicated in Figure \ref{fig:skeiningpretzels-s}, we express $V(P(-(2s+1), 2s+1, 2i+1 ))$ in terms of $V(P(-(2s+1), 2s-1, 2i+1 ))$ and of $V(T(2, 2i-2s))$. Then skeining where indicated in Figure \ref{fig:skeiningpretzelscd-s}, we express $V(P(-(2s+1), 2s-1, 2i+1 ))$ in terms of $V(P(-(2s-1), 2s-1, 2i+1 ))$ and of $V(T(2, 2i+2s))$. Composing we obtain the following.

\begin{align*}\label{form:formform}
&V(P(-(2s+1), 2s+1, 2i+1 )) - V(P(-(2s-1), 2s-1, 2i+1 )) =\\
&= t^{-1}(t^{1/2}-t^{-1/2})\bigg( V(T(2, 2i+2s)) - V(T(2, 2i-2s)) \bigg) = \\
&=  (t^{-2}-t^{-1})\bigg( t^{-2i-2s} + t^{-2i-2s+2} +  \sum_{j=0}^{2i+2s-3}(-1)^j\, t^{-j} - t^{-2i+2s} - t^{-2i+2s+2} - \sum_{j=0}^{2i-2s-3}(-1)^j\, t^{-j} \bigg)  = \\
&\underset{i>s+3/2}{=}   (t^{-2}-t^{-1})t^{-2i-1}\bigg( t^{-2s+1}   + t^{-2s+3} - t^{2s+1} - t^{2s+3} +\sum_{j=2i-2s-2}^{2i+2s-3}(-1)^{j}\, t^{2i-j+1}   \bigg)  = \\
&=  t^{-2i-1}\bigg( t^{-2s-1} - t^{-2s}  + t^{-2s+1} - t^{-2s+2} - t^{2s-1} + t^{2s}- t^{2s+1} + t^{2s+2} + \sum_{j=2i-2s-2}^{2i+2s-3}(-1)^{j}\, t^{2i-j-1}  \\
&  \qquad \qquad - \sum_{j=2i-2s-2}^{2i+2s-3}(-1)^{j}\, t^{2i-j} \bigg)  =  t^{-2i-1}\bigg( t^{-2s-1} - t^{-2s}  + t^{-2s+1} - t^{-2s+2} - t^{2s-1} + t^{2s}- t^{2s+1} + t^{2s+2} + \\
&+ \sum_{j=-2s+2}^{2s+1}(-1)^{j+1}\, t^{j}  + \sum_{j=-2s+3}^{2s+2}(-1)^{j+1}\, t^{j} \bigg)  = t^{-2i-1}\bigg(  t^{2s+1} -t^{2s} + t^{2s-1} +2 \sum_{j=-2s+2}^{2s-2}(-1)^{j+1}t^j +  t^{-2s+1}-\\
&- t^{-2s}+t^{-2s-1} \bigg) =  t^{-2i-1}\bigg(  (t^{2s+1}+t^{-(2s+1)}) -(t^{2s}+t^{-2s}) + (t^{2s-1}+  t^{-2s+1}) + 2 \sum_{j=1}^{2s-2}(-1)^{j+1}(t^j + t^{-j}) - 2 \bigg)
\end{align*}

Now we prove that setting $s=1$ in Formula \ref{form:form} we obtain the Jones polynomial for $P(-3, 3, 2i+1)$ calculated in Theorem \ref{thm:Pretzelknots}.
\begin{align*}
&1 + t^{-2i-1}\bigg[  t^{2\cdot 1+1} + t^{-(2\cdot +1)}  - ( t + t^{-1} )  - 2\cdot 1 + \sum_{j=1}^{2\cdot 1}(-1)^{j+1}\, (2\cdot 1 +1-j)\, ( t^{j} + t^{-j} ) \bigg] \\
& = 1 + t^{-2i-1}\bigg[  t^{3} + t^{-3}  - ( t + t^{-1} )  - 2\cdot 1 + 2 (t + t^{-1}) - (t^2 + t^{-2})\bigg] = \\
&= 1 +  t^{-2i+2} +  t^{-2i-4} + t^{-2i} + t^{-2i-2} - 2 t^{-2i-1} - t^{-2i+1} - t^{-2i-3} = \\
& =t^{-2i-4} - t^{-2i-3}+ t^{-2i-2}- 2 t^{-2i-1}+ t^{-2i} -  t^{-2i+1}+  t^{-2i+2} + 1
\end{align*}

Finally, we prove Formula \ref{form:form} satisfies the recursion relation above.

\begin{align*}
&1 + t^{-2i-1}\bigg[  t^{2s+1} + t^{-(2s+1)}  - ( t + t^{-1} )  - 2s + \sum_{j=1}^{2s}(-1)^{j+1}\, (2s+1-j)\, ( t^{j} + t^{-j} ) \bigg] - \\
& \qquad \qquad -1 - t^{-2i-1}\bigg[  t^{2s-1} + t^{-(2s-1)}  - ( t + t^{-1} )  - (2s-2) + \sum_{j=1}^{2s-2}(-1)^{j+1}\, (2s-1-j)\, ( t^{j} + t^{-j} ) \bigg] = \\
& = t^{-2i-1}\bigg[  t^{2s+1} + t^{-(2s+1)}   - (t^{2s} + t^{-2s}) + 2(t^{2s-1} + t^{-(2s-1)})  - ( t^{2s-1} + t^{-(2s-1)})  -2 - \\
& + \sum_{j=1}^{2s-2}(-1)^{j+1}\, 2\, ( t^{j} + t^{-j} ) \bigg] = t^{-2i-1}\bigg[  t^{2s+1} + t^{-(2s+1)}   - (t^{2s} + t^{-2s}) + (t^{2s-1} + t^{-(2s-1)})  -2 - \\
& +  2\sum_{j=1}^{2s-2}(-1)^{j+1} ( t^{j} + t^{-j} ) \bigg]
\end{align*}

This concludes the proof.

\end{proof}

\begin{cor}
Let $s$ be a positive integer. The family of knots $\big(  P(-(2s+1), 2s+1, 2i+1) \big)_{i\in \mathbf{Z}^{+}, i>s+3/2}$ is made up of distinct prime knots.
\end{cor}
\begin{proof}
This is a direct consequence of Theorem \ref{thm:Pretzelknotsinfty}.
\end{proof}

\section{Directions for further work}\label{sec:furtherwork}

The work that led to the current article began by considering that by tying a knot in a band along with a twist that compensates for the knot's writhe, then we get a new knot with exactly the same Seifert matrix as the first one. Thus,  starting with a given knot and spanning surface for it, we surely obtain infinitely many distinct knots that have exactly the same Seifert pairing as the original, and hence the same Alexander module. The technicalities with the telling apart of these knots led us to a simplification which materialized into this article. We plan to address our original idea in the near future.

For instance, in Figure \ref{fig:Seifert} we give an example of the use of Seifert's method. The reader can verify by calculating the Jones polynomial that these  knots are distinct.

\begin{figure}[!ht]
	\psfrag{P(-3, 3, 3)}{\Huge$P(-3, 3, 3)$}
	\psfrag{P(-3, 3, 2i+1)}{\Huge$P(-3, 3, 2i+1)$}
	\psfrag{2i+1}{\huge$2i+1$}
	\psfrag{eq}{\Huge$c=2b-a$}
	\centerline{\scalebox{.5}{\includegraphics{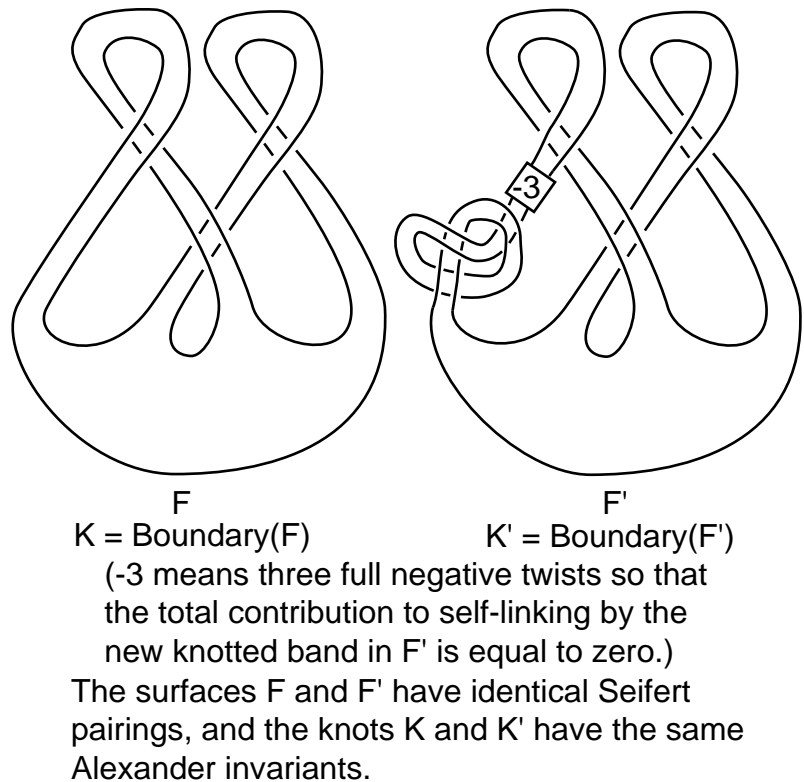}}}
	\caption{The Seifert method.}\label{fig:Seifert}
\end{figure}

\section*{Acknowledgments}
P.L. acknowledges support from FCT (Funda\c c\~ao para a Ci\^encia e a Tecnologia), Portugal, through project FCT EXCL/MAT-GEO/0222/2012, ``Geometry and Mathematical Physics''.


\begin{thebibliography}{99}

\bibitem{AitchisonSilver}
I. Aitchison,  D. Silver,  \emph{On certain fibred ribbon disc pairs}, Trans. Amer. Math. Soc. {\bf 306},  2 (1988) 529–-551




\bibitem{Friedl} S. Friedl,
\emph{Realizations of Seifert matrices by hyperbolic knots}, J. Knot Theory Ramifications {\bf 18}, 11 (2007),  1471--1474.


\bibitem{HittSilver}
L. Hitt,  D. Silver,  \emph{Ribbon knot families via Stallings' twists} J. Austral. Math. Soc. Ser. A {\bf 50}, 3 (1991) 356-–372.


\bibitem{Kawauchi-Pretzels} A. Kawauchi,
\emph{Classification of pretzel knots}, Kobe J. Math. {\bf 2} (1985),  11--22.

\bibitem{KL} D. Kim, J. Lee,
\emph{Some invariants of pretzel links}, Bull. Austral. Math. Soc. {\bf 75} (2007),  253--271.


\bibitem{Landvoy} R. Landvoy,
The Jones polynomial of pretzel knots and links, {\it Topol. Appl}. {\bf 83} (1998), 135--147.


\bibitem{Lickorish}
        W. B. R. Lickorish, \emph{An introduction to knot theory}, Graduate Texts in Mathematics
         {\bf175}, Springer Verlag, New York (1997)


\bibitem{Livingston1981}
C. Livingston, \emph{Homology cobordisms of 3-manifolds, knot concordances, and prime knots}, Pacific J. of Math. {\bf 94} (1981) 193--206


\bibitem{Livingston2002}
C. Livingston, \emph{Seifert forms and concordance}, Geom. Topol., {\bf 6} (2002) 403--408



\bibitem{Morton} H. Morton,
\emph{Infinitely many fibered knots having the same Alexander polynomial}, Topology {\bf 17} (1978),  101--041.


\bibitem{Parris} R. L. Parris, Pretzel knots, Ph. D. Thesis, Princeton University (1978).


\bibitem{Rapaport}
E. S. Rapaport,  \emph{On the commutator subgroup of a knot group}, Ann. of Math. (2) {\bf 71} (1960) 157--162

\bibitem{Seifert}
H. Seifert, \emph{\"{U}ber das Geschlecht von Knotten }, Math. Annalen, {\bf 110} (1934), pp. 571--592.

\bibitem{Sto}
A. Stoimenow, \emph{Realizing Alexander polynomials by hyperbolic links}, Expo. Math. {\bf 28} (2010), 133--178.

\end{thebibliography}
\end{document}